\documentclass[12pt]{amsart}
\usepackage[T1]{fontenc}
\usepackage[latin1]{inputenc}
\usepackage{typearea}
\usepackage{geometry}
\usepackage{ulem}
\usepackage{amsmath}
\usepackage{amssymb}
\usepackage{latexsym}
\usepackage{enumerate}
\usepackage{amsthm}
\usepackage[all]{xy}
\usepackage{hhline}
\usepackage{epsf} %fuer bildchen
\usepackage{cite}
\usepackage{verbatim}
\usepackage{mathtools}
\usepackage{comment}
\usepackage{dsfont}
\usepackage{mathrsfs}

\usepackage{color}

\newtheorem{theorem}{{\sc Theorem}}[section]
\newtheorem{cor}[theorem]{{\sc Corollary}}

\newtheorem{lemma}[theorem]{{\sc Lemma}}

\theoremstyle{remark}
\newtheorem{remark}[theorem]{{\sc Remark}}

\theoremstyle{definition}

\newcommand{\R}{\mathbb{R} }
\newcommand{\N}{\mathbb{N} }

\newcommand{\A}{\mathcal{A}}
\newcommand{\B}{\mathcal{B}}
\newcommand{\F}{\mathcal{F}}
\newcommand{\G}{\mathcal{G}}

\newcommand{\D}{\mathcal{D}}
\newcommand{\W}{\mathcal{W}}
\newcommand{\K}{\mathcal{K}}

\newcommand{\Prob}{\mathbb{P}}
\newcommand{\E}{\mathbb{E}}

\newcommand{\Om}{\Omega}

\providecommand{\babs}[1]{\bigl\lvert #1\bigr\rvert}
\providecommand{\Babs}[1]{\Bigl\lvert #1\Bigr\rvert}

\DeclareMathOperator{\Var}{Var}

\DeclareMathOperator{\Cov}{Cov}

\DeclareMathOperator{\Lip}{Lip}

\renewcommand{\phi}{\varphi}
\renewcommand{\epsilon}{\varepsilon}

\renewcommand{\rho}{\varrho}
\renewcommand{\P}{\Prob}

\begin{document}
\title[Berry-Esseen bound]{The Berry-Esseen bound in de Jong's CLT}
\author{Christian D\"obler}
\thanks{\noindent Mathematisches Institut der Heinrich Heine Universit\"{a}t D\"usseldorf\\
Email: christian.doebler@hhu.de }
\begin{abstract}  
We prove a Berry-Esseen bound in de Jong's classical CLT for normalized, completely degenerate $U$-statistics, which says that the convergence of the fourth moment sequence to three and a Lindeberg-Feller type negligibility condition are sufficient for asymptotic normality. Our bound is of the same optimal order as the bound on the Wasserstein distance to normality that has recently been proved by D\"obler and Peccati (2017).   
\end{abstract}

\maketitle

\section{Introduction}\label{intro}
\subsection{Motivation and overview}
Let $p$ be a fixed positive integer and suppose that, for each $n\in\N$, the random variable $W_n$ is a normalized, completely degenerate and not necessarily symmetric $U$-statistic of order $p$, based on independent random variables $X_1,\dotsc,X_n$ (see Section \ref{mainresult} for precise definitions), defined on a probability space that might vary with $n$.
Henceforth, we will sometimes just refer to such a quantity as $W_n$ as a (normalized) \textit{degenerate $U$-statistic of order $p$}. 

In the seminal paper \cite{deJo90} (see also \cite{deJo87, deJo89}), P. de Jong proved the following remarkable CLT: If 
\begin{equation}\label{fmcond}
\lim_{n\to\infty}\E[W_n^4]=3
\end{equation}
and a Lindeberg-Feller type negligibility condition for the sequence $(W_n)_{n\in\N}$ is satisfied (see again Section \ref{mainresult} for a precise statement), then $(W_n)_{n\in\N}$ converges in distribution to a standard normal random variable $Z$.

In view of the typically non-normal limiting distributions of (symmetric) degenerate $U$-statistics with a fixed kernel \cite{Serf80, RubVi80, Greg77, DynMan83}, which does not depend on the sample size $n$, de Jong's theorem is a quite surprising result. Moreover, the class of degenerate $U$-statistics is rather large, containing for example the so-called \textit{homogeneous sums} or \textit{homogeneous multilinear forms} in independent centered real random variables with unit variance (see Section \ref{mainresult}).  \medskip

In the recent paper \cite{DP17}, G. Peccati and the author applied the \textit{exchangeable pairs approach} within \textit{Stein's method} \cite{St72, St86}
to prove a quantitative version of de Jong's purely qualitative statement, by providing an explicit error bound on the \textit{Wasserstein distance}
\[d_\W(W_n,Z):=\sup_{h\in\Lip(1)}\babs{\E[h(W_n)]-\E[h(Z)]}\]
 between the distribution of such a normalized, degenerate $U$-statistic $W_n$ and the standard normal distribution of $Z$. Here, $\Lip(1)$ is the class of all Lipschitz functions on $\R$ with Lipschitz constant $1$. The Wasserstein bound from \cite{DP17} reads
\begin{align}\label{wassbound}
d_\W(W_n,Z) &\leq \Biggl(\sqrt{\frac{2}{\pi}}+\frac{4}{3}\Biggr)\sqrt{\babs{\E[W_n^4]-3}}+\sqrt{\kappa_p}\Biggl(\sqrt{\frac{2}{\pi}}+ \frac{2\sqrt{2}}{\sqrt{3}}\Biggr)\rho_{n},
\end{align}
where the quantity $\rho_{n}$, which encodes the Lindeberg-Feller condition, is defined in Section \ref{mainresult} below and where $\kappa_p$ is a finite combinatorial constant that only depends on $p$.

The goal of the present note is to complement the Wasserstein bound \eqref{wassbound} with a \textit{Berry-Esseen bound}, that is, a bound on the \textit{Kolmogorov distance} 
\[d_\K(W_n,Z):=\sup_{t\in\R}\babs{\P(W_n\leq t)-\P(Z\leq t)}=\sup_{t\in\R}\babs{\P(W_n\leq t)-\Phi(t)}\]
between the distribution of $W_n$ and the standard normal distribution, which is of the same order as the bound \eqref{wassbound}. Here, and in what follows,
\[\Phi(t)=\frac{1}{\sqrt{2\pi}}\int_{-\infty}^t e^{-x^2/2}dx, \quad t\in\R,\]
denotes the standard normal distribution function. From a statistical viewpoint, error bounds on the Kolmogorov distance are usually more informative and useful than bounds on the Wasserstein distance. For instance, if $W_n$ is the test statistic of an asymptotic test, then $d_\K(W_n,Z)$ is the maximal error that arises from working with the standard normal quantiles instead of the true ones of $W_n$. However, when applying Stein's method, it is in general considerably more difficult to obtain sharp bounds on the Kolmogorov distance than on the Wasserstein distance. This is roughly because the solutions to the Stein equation for the Kolmogorov distance lack one order of smoothness as compared to those for the Wasserstein distance. Although, for a standard normal $Z$ and 
 for an integrable real random variable $X$, on has the general inequality 
\[d_\K(X,Z)\leq\sqrt{d_\W(X,Z)},\]    
this inequality usually does not yield sharp bounds since, for a normal limit, the actual rates of convergence in the Kolmogorov distance and in the Wasserstein distance are typically the same. 

As in \cite{DP17}, the proof of our Berry-Esseen result relies on a combination of the exchangeable pairs approach in Stein's method with the theory of \textit{Hoeffding decompositions} \cite{Hoeffding} of arbitrary functionals on product probability spaces. In particular, we rely here on several crucial identities and bounds that were proved in \cite{DP17} in the context of Hoeffding decompositions and exchangeable pairs (see again Section \ref{mainresult} for details).
However, in place of Stein's classical exchangeable pairs bound on the Wasserstein distance to normality (see \cite[Lecture 3, Theorem 1]{St86}, we employ here a recent bound on the Kolmogorov distance to normality from \cite{ShZh}.

\subsection{Further related references}
In addition to the bound on the Wasserstein distance, the work \cite{DP17} also provided quantitative multivariate extensions of de Jong's CLT for vectors of such degenerate $U$-statistics. In particular, for vectors of degenerate $U$-statistics with pairwise different orders, these error bounds entail conditions for the multivariate CLT to hold that are equivalent to the conglomeration of the conditions for the univariate CLTs for the individual components to be valid. The related paper \cite{DP18b} complemented the quantitative CLT from \cite{DP17} by proving analogous error bounds for the (centered) Gamma approximation of such a degenerate $U$-statistic and the very recent paper \cite{DKP22a} even provided multivariate functional versions of de Jong type CLTs. We refer to these references for pointers to the relevant literature, example cases and possible applications. \smallskip

From a modern perspective, de Jong's CLT can be considered the historically first instance of a so-called \textit{fourth moment theorem}, which generally states that a normalized sequence $(W_n)_{n\in\N}$ of real random variables is asymptotically normally distributed, if \eqref{fmcond} is satisfied, possibly up to imposing some further negligibility condition that sometimes cannot be dispensed with. In particular, de Jong's result preceded the remarkable uni- and multivariate fourth moment theorems for Gaussian Wiener chaos \cite{NuaPec05,NouPec09a,PeTu}, for multiple Wiener-It\^{o} integrals on Poisson spaces \cite{DP18a, DP18b, DVZ}, for multiple integrals with respect to a general Rademacher sequence \cite{DK19, Zheng} as well as for eigenfunctions of Markov diffusion operators \cite{ledoux, ACP,CNPP}.   
\smallskip
\begin{comment}
We briefly mention here that the quantitative versions of the fourth moment theorems for Gaussian \cite{NouPec09a}, Poisson \cite{DP18a, DVZ} and for Rademacher chaos \cite{DK19, Zheng} have all been proved by suitable variants of the of the so-called \textit{Malliavin-Stein method}, initiated in the seminal paper \cite{NouPec09a}. This method, in its original form, relies on a fruitful combination of Stein's method of normal (and non-normal) approximation and the \textit{Malliavin calculus} from stochastic analysis and has become an utterly active field of research in recent years. A comprehensive list of publications in the realm of this method can be found on the regularly updated website 
\begin{center}
\verb| https://sites.google.com/site/malliavinstein| 
\end{center}
maintained by I. Nourdin.\medskip
\end{comment}

As it turned out, the methodology developed in \cite{DP17} is rather flexible and has been successfully adapted to prove error bounds on the (multivariate) normal approximation in other situations as well: In the paper \cite{DP19}, by combining the general approach from \cite{DP17} with a new formula for the product of two symmetric, degenerate $U$-statistics, new error bounds on the normal approximation of symmetric (not necessarily degenerate) $U$-statistics with possibly sample size dependent kernels were proved. These analytical bounds only involve powers of the sample size $n$ and integral norms of so-called \textit{contraction kernels} associated to the symmetric $U$-statistic. These techniques were further refined in the reference \cite{DKP22b}, where we proved general functional CLTs for symmetric $U$-statistics with sample size dependent kernels. In the paper \cite{Doe23} the methodology from \cite{DP19} was further extended in order to prove accurate Wasserstein bounds on the normal approximation of quite arbitrary (symmetric and non-symmetric) functionals on product spaces.  \medskip

The remainder of this paper is structured as follows. In Section \ref{mainresult} we introduce the technical framework and state the main result of this work, Theorem \ref{maintheo}, as well as a corollary dealing with symmetric $U$-statistics. The proof of Theorem \ref{maintheo} is given in Section \ref{proof}.

\section{Framework and main result}\label{mainresult}
In what follows, let $p\leq n$ be positive integers and suppose that $X_1,\dotsc,X_n$ are independent random variables on a probability space $(\Om,\F,\P)$, assuming values in the respective measurable spaces  
$(E_1,\mathcal{E}_1),\dotsc,(E_n,\mathcal{E}_n)$. Further, let
\begin{equation*}
 f:\prod_{j=1}^n E_j\rightarrow\R\quad\text{be}\quad\bigotimes_{j=1}^n\mathcal{E}_j-\B(\R)\text{ - measurable}
\end{equation*}
in such a way that 
\[W:=W_n:=f(X_1,\dotsc,X_n)\in L^1(\P).\]
Then, as is well-known (see e.g.\cite{vitale, Major}), $W$ has a $\P$-a.s. unique decomposition, the \textit{Hoeffding decomposition}, of the form
\begin{equation}\label{HD}
W=\sum_{J\subseteq [n]} W_J,
\end{equation}  
where we write $[n]:=\{1,\dotsc,n\}$ and where the $W_J$, $J\subseteq[n]$, are random variables with the following properties:
\begin{enumerate}[(a)]
\item For each $J\subseteq[n]$ $W_J$ is $\sigma(X_i,i\in J)$-measurable. 
\item For all $J,K\subseteq[n]$ one has that $\E[W_J\mid X_i,i\in K]=0$ unless $J\subseteq K$. 
\end{enumerate}
Note that (a) implies that there are $\bigotimes_{j\in J}\mathcal{E}_j-\B(\R)$-measurable \textit{kernel functions} $\psi_J:\prod_{j\in J}E_j\rightarrow\R$ such that $W_J=\psi_J(X_i,i\in J)$, $J\subseteq[n]$. Here, and in what follows, the arguments of $\psi_J$ are plugged in according to the usual ascending order on $[n]$. Since the summands in \eqref{HD} are explicitly given by
\[W_J=\sum_{L\subseteq J}(-1)^{|J|-|L|}\,\E\bigl[W\,\bigl|\,X_i,i\in L\bigr], \quad J\subseteq[n],\]
one infers immediately that $W\in L^q(\P)$ implies that  $W_J\in L^q(\P)$ for all $J\subseteq[n]$, where $q\in[1,\infty]$. If $W\in L^2(\P)$, then its individual \textit{Hoeffding components} $W_J$, $J\subseteq[n]$, are automatically uncorrelated, making the Hoeffding decomposition a very useful tool for the analysis of variances of functionals on product spaces.

In the above setting, the random variable $W$ is called a \textit{completely degenerate, not necessarily symmetric $U$-statistic of order $p$} or \textit{degenerate $U$-statistic} for short, if its Hoeffding decomposition \eqref{HD} has the form 
\begin{equation}\label{HDdeg}
W=\sum_{J\in\D_p(n)}W_J= \sum_{J\in\D_p(n)}\psi_J(X_i,i\in J),
\end{equation} 
where we let 
\[\D_p(n):=\{J\subseteq[n]\,:\, |J|=p\}\]
denote the collection of $p$-subsets of $[n]$ , that is, if $W_K=0$ $\P$-a.s. unless $|K|=p$. In \cite{deJo90} such a $W$ is also referred to as a \textit{generalized multilinear form} since this class contains the class of random variables $Y$ of the form 
\[Y=\sum_{J\in\D_p(n)} a_J\prod_{i\in J} Y_i,\]
where $Y_1,\dotsc,Y_n$ are independent and centered real random variables and $(a_J)_{J\in\D_p(n)}$, is a family of real coefficients. Such random variables $Y$ are called \textit{homogeneous multilinear forms} or \textit{homogeneous sums} in the literature. \smallskip

If, in fact, the underlying random variables $X_1,\dotsc,X_n$ are i.i.d. and, in particular, assume values in the same space $(E,\mathcal{E})$ and the kernels $\psi_J$, $J\in\D_p(n)$, of $W$ as in \eqref{HDdeg} are all equal to the same symmetric function $\psi:E^p\rightarrow\R$, which might still depend on the sample size $n$, then $W$ is a called a \textit{completely degenerate symmetric $U$-statistic of order $p$}. Note that the symmetric kernel $\psi$ is then \textit{canonical} in the sense that 
\[\int_E \psi(x_1,\dotsc,x_{p-1},y)d\mu(y)=0\,\quad (x_1,\dotsc,x_{p-1})\in E^{p-1},\]
where $\mu$ denotes the distribution of $X_1$.  \smallskip

From now on, we will assume that $W\in L^4(\P)$ is a completely degenerate, not necessarily symmetric $U$-statistic of order $p\geq1$ and with the Hoeffding decomposition \eqref{HDdeg}, based on $X_1,\dotsc,X_n$. 
Then, $\E[W]=0$ and we will further assume that $\Var(W)=1$. Moreover, we will write 
\begin{align*}
\sigma_J^2:=\E[W_J^2]=\Var(W_J),\quad J\in\D_p(n),\quad\text{and}\quad \rho_n^2:=\max_{1\leq i\leq n} \sum_{J\in\D_p(n): i\in J}\sigma_J^2.
\end{align*} 
Note that since, by orthogonality of the Hoeffding components, 
\[1=\Var(W)=\sum_{J\in\D_p(n)} \sigma_J^2,\]
$\rho_n^2$ measures the \textit{maximal influence} that an individual random variable $X_i$ can possibly have on the total variance of $W$. We will refer to $\rho_n^2$ as a \textit{Lindeberg-Feller type quantity}.

The purpose of this note is to complement the bound \eqref{wassbound} with the following bound on the Kolmogorov distance between the law of $W$ and the standard normal distribution.

\begin{theorem}\label{maintheo}
Let $W\in L^4(\P)$ be a completely degenerate, not necessarily symmetric $U$-statistic of order $p$, based on the independent random variables $X_1,\dotsc,X_n$, where $p\leq n$, and suppose that $\Var(W)=\E[W^2]=1$. Then, it holds that 
\begin{align*}
\sup_{t\in\R}\babs{\P(W\leq t)-\Phi(t)}&%=d_\K(W,Z)
\leq 11.9\sqrt{\babs{\E[W^4]-3}} +\bigl(3.5+10.8\sqrt{\kappa_p}\bigr)\rho_n,
\end{align*} 
where $\kappa_p\in (0,\infty)$ is a combinatorial constant that only depends on $p$.% and where $
\end{theorem}

For a symmetric, completely degenerate $U$-statistic $W\in L^4(\P)$, using that $\rho_n^2=p/n$ and that, as we have observed in the recent article \cite{Doe23}, the choice $\kappa_p=2p$ is possible in this case, we directly infer the following result. 
\begin{cor}\label{cor1}
Suppose that $W\in L^4(\P)$ is a completely degenerate, symmetric $U$-statistic of order $p$, based on the independent random variables $X_1,\dotsc,X_n$, where $p\leq n$, and suppose that $\Var(W)=\E[W^2]=1$. Then, one has the bound
\begin{align*}
\sup_{t\in\R}\babs{\P(W\leq t)-\Phi(t)}&%=d_\K(W,Z)
\leq 12\sqrt{\babs{\E[W^4]-3}} +19\frac{p}{\sqrt{n}}.
\end{align*} 
\end{cor}

\begin{remark}\label{mtrem}
 \begin{enumerate}[(a)]
  \item The bound in Theorem \ref{maintheo} is of the same order as the Wasserstein bound \eqref{wassbound}. As can be seen from simple examples, like sums of independent symmetric Rademacher random variables, it is sharp, in general. 
  \item As has been shown in \cite[Theorem 1.6]{DK19} in the context of homogeneous sums based on independent symmetric Rademacher random variables, for $p\geq2$ one cannot, in general, dispense with the Lindeberg-Feller type condition $\lim_{n\to\infty}\rho_n^2=0$ to obtain a CLT, i.e. the fourth moment condition $\lim_{n\to\infty}\E[W_n^4]=3$ alone is not sufficient.
 \end{enumerate}

\end{remark}

\section{Proof of Theorem \ref{maintheo}}\label{proof}
In this section we will prove Theorem \ref{maintheo} by employing a recent Berry-Esseen bound for exchangeable pairs taken from \cite{ShZh}. We first recall the construction of the exchangeable pair $(W,W')$ from \cite{DP17}: \medskip

Let $Y=(Y_1,\dotsc,Y_n)$ be an independent copy of $X=(X_1,\dotsc,X_n)$ and suppose that $\alpha$ is a uniformly distributed random index with values in $[n]$, which is also independent of $X$ and $Y$. Then, letting $X':=(X_1',\dotsc,X_n')$ be given by $X_i'=X_i$ for $i\not=\alpha$ and $X_i'=Y_i$ for $i=\alpha$, we have that the pair $(X,X')$ is exchangeable, i.e. has the same distribution as the pair $(X',X)$. Recalling that $W=f(X)=f(X_1,\dotsc,X_n)$ is a functional of $X$, we can thus let $W':=f(X')=f(X_1',\dotsc,X_n')$ and obtain an exchangeable pair $(W,W')$ of real random variables. In \cite[Lemma 2.3]{DP17} it was shown that the pair $(W,W')$ satisfies \textit{Stein's linear regression property} with $\lambda=p/n$, i.e. 
\begin{equation}\label{linreg}
 \E\bigl[W'-W\,\bigl|\,W\bigr]=\E\bigl[W'-W\,\bigl|\,X\bigr]=-\frac{p}{n}W\,.
\end{equation}  
For the proof of Theorem \ref{maintheo} we will need the following further auxiliary results from \cite{DP17}. 

\begin{lemma}[Lemma 2.11 of \cite{DP17}]\label{le1} 
 For the above constructed exchangeable pair we have 
 \begin{equation*}
\Var\Bigl(\frac{n}{2p}\E\bigl[(W'-W)^2\,\bigl|\,X\bigr]\Bigr)\leq \E\bigl[W^4\bigr]-3+\kappa_p\rho_n^2,
\end{equation*}
where $\kappa_p\in(0,\infty)$ only depends on $p$.
\end{lemma}

\begin{lemma}[Lemma 2.12 of \cite{DP17}]\label{le2} 
 For the above constructed exchangeable pair we have the bound
 \[\frac{n}{4p}\E\bigl[(W'-W)^4\bigr]\leq 2\bigl(\E[W^4]-3\bigr)+3\kappa_p \rho_n^2\,,\]
where $\kappa_p$ is the same as in Lemma \ref{le1}.
\end{lemma}

We will further make use of the next result, which we derive from \cite[Lemmas 2.2 and 2.7]{DP17}.
\begin{lemma}\label{le3}
For the above constructed exchangeable pair we have 
 \begin{align*}
\E\Bigl[W^2\frac{n}{2p}\E\bigl[(W'-W)^2\,\bigl|\,W\bigr]\Bigr]&\leq\E\bigl[W^4\bigr],\\
\frac{n}{4p}\E\bigl[(W'-W)^4\bigr]&\leq2\E\bigl[W^4\bigr].
\end{align*}
\end{lemma}

\begin{proof}
By \cite[Lemma 2.7]{DP17},
letting
\[W^2=\sum_{\substack{M\subseteq[n]:\\ |M|\leq 2p}} U_M\]
denote the Hoeffding decomposition of $W^2$, one has that the Hoeffding decomposition of $\frac{n}{2p}\E\bigl[(W'-W)^2\,\bigl|\,X\bigr]$ is given by 
\[ \frac{n}{2p}\E\bigl[(W'-W)^2\,\bigl|\,X\bigr]=\sum_{\substack{M\subseteq[n]:\\ |M|\leq 2p-1}}\frac{2p-|M|}{2p} U_M.\]
Hence, using the orthogonality of Hoeffding components in $L^2(\P)$, we obtain that 
\begin{align*}
&\E\Bigl[W^2\frac{n}{2p}\E\bigl[(W'-W)^2\,\bigl|\,W\bigr]\Bigr]=\E\Bigl[W^2\frac{n}{2p}\E\bigl[(W'-W)^2\,\bigl|\,X\bigr]\Bigr]\\
&=\sum_{\substack{M\subseteq[n]:\\ |M|\leq 2p-1}}\frac{2p-|M|}{2p}\E[U_M^2]\leq\E[W^2]^2+\frac{2p-1}{2p}\sum_{\substack{M\subseteq[n]:\\1\leq |M|\leq 2p-1}}\Var(U_M)\\
&\leq \E[W^2]^2+\sum_{\substack{M\subseteq[n]:\\1\leq |M|\leq 2p}}\Var(U_M)=\E[W^2]^2+\Var(W^2)=\E[W^4],
\end{align*}
where we have used that $U_\emptyset=\E[W^2]$ in the first inequality. This proves the first claim. For the second claim, we just note that by \cite[Lemma 2.2]{DP17} we have that 
\[\frac{n}{4p}\E\bigl[(W'-W)^4\bigr]\leq3\E\Bigl[W^2\frac{n}{2p}\E\bigl[(W'-W)^2\,\bigl|\,W\bigr]\Bigr]-\E\bigl[W^4\bigr]\]
and apply the bound just proven.
 \end{proof}

We remark that, by homogeneity, Lemma \ref{le3} continues to hold when $\Var(W)=\E[W^2]\not=1$. We are now ready to prove our main result.

\begin{proof}[Proof of Theorem \ref{maintheo}]
In view of \eqref{linreg}, by \cite[Theorem 2.1]{ShZh} we have the bound 
\begin{align}\label{pr1}
\sup_{t\in\R}\babs{\P(W\leq t)-\Phi(t)}&\leq \E\Babs{1-\frac{n}{2p}\E\bigl[(W'-W)^2\,\bigl|\,W\bigr]}\notag\\
&\;+\frac{n}{p}\E\Babs{\E\bigl[|W'-W|(W'-W)\,\bigl|\,W\bigr]}.
\end{align}
Since 
\[\E[(W'-W)^2]=\frac{2p}{n},\]
for the first term on the right hand side of \eqref{pr1}, from Lemma \ref{le1} we see that 
\begin{align}\label{pr2}
&\E\Babs{1-\frac{n}{2p}\E\bigl[(W'-W)^2\,\bigl|\,W\bigr]}\leq\biggl(\Var\Bigl(\frac{n}{2p}\E\bigl[(W'-W)^2\,\bigl|\,W\bigr]\Bigr)\biggr)^{1/2}\notag\\
&\leq \biggl(\Var\Bigl(\frac{n}{2p}\E\bigl[(W'-W)^2\,\bigl|\,X\bigr]\Bigr)\biggr)^{1/2}
\leq \sqrt{|\E\bigl[W^4\bigr]-3|}+\sqrt{\kappa_p}\rho_n,
\end{align}
where we have used the inequality $\Var(\E[T|\G])\leq \Var(\E[T|\A])$ for sub-$\sigma$-fields $\G\subseteq\A$ of $\F$ and $T\in L^2(\P)$.

To deal with the second term, we first introduce some useful notation. Let $\theta:\R\rightarrow\R$ be the function given by $\theta(x):=|x|x$. Moreover, for a random variable 
$T=t(X_1,\dotsc,X_n)$ and $i,j\in[n]$ with $i\not=j$ we let 
\[T^{(i)}:=t(X_1,\dotsc,X_{i-1},Y_i,X_{i+1},\dotsc,X_n)\]
and
\[T^{(i,j)}:=(T^{(i)})^{(j)}:=(T^{(j)})^{(i)}:=t(X_1,\dotsc,X_{i-1},Y_i,X_{i+1},\dotsc,X_{j-1},Y_j,X_{j+1},\dotsc,X_n),\]
that is, we replace the respective components of the vector $X$ with those from $Y$ in the argument of the function $t$. Furthermore, using this notation, we let 
\[D_i:=W^{(i)}-W=\sum_{J\in\D_p(n)}\bigl(W_J^{(i)}-W_J\bigr)=\sum_{\substack{J\in\D_p(n):\\ i\in J}}\bigl(W_J^{(i)}-W_J\bigr),\quad i\in[n].\]
With this notation at hand, using independence of $\alpha, X$ and $Y$, we have 
\begin{align*}
&n\E\bigl[|W'-W|(W'-W)\,\bigl|\,X\bigr]=\sum_{i=1}^n\E\bigl[|D_i| D_i\,\bigl|\, X\bigr]= \sum_{i=1}^n\E\bigl[\theta(D_i)\,\bigl|\, X\bigr]
\end{align*}
and since each $\theta(D_i)$ has a symmetric distribution and, hence, $\E[\theta(D_i)]=0$, it follows that 
\begin{align}\label{pr3}
&n\E\Babs{\E\bigl[|W'-W|(W'-W)\,\bigl|\,W\bigr]}\leq n\E\Babs{\E\bigl[|W'-W|(W'-W)\,\bigl|\,X\bigr]}\notag\\
&=\E\Babs{\sum_{i=1}^n\E\bigl[\theta(D_i)\,\bigl|\, X\bigr]}\leq\Biggl(\Var\biggl(\sum_{i=1}^n\E\bigl[\theta(D_i)\,\bigl|\, X\bigr]\biggr)\Biggr)^{1/2}\notag\\
&=\Biggl(\sum_{i=1}^n\Var\Bigl(\E\bigl[\theta(D_i)\,\bigl|\, X\bigr]\Bigr)+\sum_{i\not=j}\Cov\Bigl(\E\bigl[\theta(D_i)\,\bigl|\, X\bigr], \E\bigl[\theta(D_j)\,\bigl|\, X\bigr]\Bigr)\Biggr)^{1/2}.
\end{align}
For the sum of variances, using Lemma \ref{le2} as well as $\E[\theta(D_i)]=0$ and the conditional Jensen inequality we have that 
\begin{align}\label{pr4}
&\sum_{i=1}^n\Var\Bigl(\E\bigl[\theta(D_i)\,\bigl|\, X\bigr]\Bigr)=\sum_{i=1}^n\E\biggl[\Bigl(\E\bigl[\theta(D_i)\,\bigl|\, X\bigr]\Bigr)^2\biggr]
\leq\sum_{i=1}^n\E\Bigl[\bigl(\theta(D_i)\bigr)^2\Bigr]\notag\\
&=\sum_{i=1}^n\E\bigl[D_i^4\bigr]=n\E\bigl[(W'-W)^4\bigr]
\leq 8p\bigl(\E[W^4]-3\bigr)+12p\kappa_p \rho_n^2.
\end{align}
In order to deal with the sum of covariances, first note that, by the total covariance law, we have for $i\not=j$ that
\begin{align*}
&\Cov\bigl(\theta(D_i),\theta(D_j)\bigr)\\
&=\Cov\Bigl(\E\bigl[\theta(D_i)\,\bigl|\, X\bigr], \E\bigl[\theta(D_j)\,\bigl|\, X\bigr]\Bigr)
+\E\Bigl[\Cov\bigl(\theta(D_i),\theta(D_j)\,\bigl|\,X\bigr)\Bigr]\\
&=\Cov\Bigl(\E\bigl[\theta(D_i)\,\bigl|\, X\bigr], \E\bigl[\theta(D_j)\,\bigl|\, X\bigr]\Bigr),
\end{align*}
since, given $X$, $\theta(D_i)$ is a (measurable) function of $Y_i$, whereas $\theta(D_j)$ is a function of $Y_j$ and, hence, the two are conditionally independent given $X$. In particular, 
\[\Cov\bigl(\theta(D_i),\theta(D_j)\,\bigl|\,X\bigr)=0\quad\P\text{-a.s.}\]
Hence, using again that $\E[\theta(D_i)]=0$, we have 
\begin{align}\label{pr5}
&\Cov\Bigl(\E\bigl[\theta(D_i)\,\bigl|\, X\bigr], \E\bigl[\theta(D_j)\,\bigl|\, X\bigr]\Bigr)=\Cov\bigl(\theta(D_i),\theta(D_j)\bigr)=\E\bigl[\theta(D_i)\theta(D_j)\bigr]
\end{align}
and we further make the fundamental claim that for the latter term the identity
\begin{align}\label{pr11}
%&\Cov\Bigl(\E\bigl[\theta(D_i)\,\bigl|\, X\bigr], \E\bigl[\theta(D_j)\,\bigl|\, X\bigr]\Bigr)=\Cov\bigl(\theta(D_i),\theta(D_j)\bigr)
\E\bigl[\theta(D_i)\theta(D_j)\bigr]%\notag\\
&=\frac{1}{4}\E\Bigl[\bigr(\theta(D_i^{(j)})-\theta(D_i)\bigr)\cdot\bigr(\theta(D_j^{(i)})-\theta(D_j)\bigr)\Bigr]
\end{align}
holds true. This identity is one of the main observations for our proof to succeed. 
%where, for the last equality, which is a fundamental observation for the whole proof, we have used that 
In order to prove it, we make sure that
\begin{align*}
\E\Bigl[\theta(D_i^{(j)})\theta(D_j)\Bigr]=-\E\bigl[\theta(D_i)\theta(D_j)\bigr]=-\E\Bigl[\theta(D_i^{(j)})\theta(D_j^{(i)})\Bigr].
\end{align*}
These identities follow from independence and (anti-)symmetry by interchanging, respectively, the identically distributed variables $Y_j$ and $X_j$ in the expectation
\begin{align*}
&\E\Bigl[\theta(D_i^{(j)})\theta(D_j)\Bigr]=\E\Bigl[\bigl(W^{(i,j)}-W^{(j)}\bigr)\babs{W^{(i,j)}-W^{(j)}}\bigl(W^{(j)}-W\bigr)\babs{W^{(j)}-W}\Bigr]\\
&=\E\Bigl[\bigl(W^{(i)}-W\bigr)\babs{W^{(i)}-W}\bigl(W-W^{(j)}\bigr)\babs{W^{(j)}-W}\Bigr]\\
&=-\E\bigl[\theta(D_i)\theta(D_j)\bigr]
\end{align*}
and the identically distributed pairs $(Y_i,Y_j)$ and $(X_i,X_j)$ in the expectation
\begin{align*}
&\E\Bigl[\theta(D_i^{(j)})\theta(D_j^{(i)})\Bigr]\\
&=\E\Bigl[\bigl(W^{(i,j)}-W^{(j)}\bigr)\babs{W^{(i,j)}-W^{(j)}}\bigl(W^{(j,i)}-W^{(i)}\bigr)\babs{W^{(j,i)}-W^{(i)}}\Bigr]\\
&=\E\Bigl[\bigl(W-W^{(i)}\bigr)\babs{W-W^{(i)}}\bigl(W-W^{(j)}\bigr)\babs{W^{(j)}-W}\Bigr]\\
&=\E\bigl[\theta(D_i)\theta(D_j)\bigr].
\end{align*}
Now, as observed in display (4.15) in \cite{PTh} for instance, by using a Taylor argument, one has  
\[\bigl(\theta(y)-\theta(x)\bigr)^2\leq 8x^2(y-x)^2+2(y-x)^4,\quad x,y\in\R,\]
so that \eqref{pr5}, \eqref{pr11} and the inequality $|ab|\leq a^2/2+b^2/2$ imply that 
\begin{align}\label{pr6}
&\Cov\Bigl(\E\bigl[\theta(D_i)\,\bigl|\, X\bigr], \E\bigl[\theta(D_j)\,\bigl|\, X\bigr]\Bigr)\notag\\
&\leq\frac18\E\Bigl[\bigr(\theta(D_i^{(j)})-\theta(D_i)\bigr)^2\Bigr]+\frac18\E\Bigl[\bigr(\theta(D_j^{(i)})-\theta(D_j)\bigr)^2\Bigr]\notag\\
&\leq\E\Bigl[D_i^2\bigl(D_i^{(j)}-D_i\bigr)^2\Bigr]+\frac14\E\Bigl[\bigl(D_i^{(j)}-D_i\bigr)^4\Bigr]\notag\\
&\;+\E\Bigl[D_j^2\bigl(D_j^{(i)}-D_j\bigr)^2\Bigr]+\frac14\E\Bigl[\bigl(D_j^{(i)}-D_j\bigr)^4\Bigr].
\end{align}
To proceed from here, we make the next important observation that, for fixed $i\in[n]$, the random variable $D_i$ is again a completely degenerate $U$-statistic of order $p$, this time based on the $n+1$ independent random variables $X_1,\dotsc,X_n,X_{n+1}:=Y_i$. Indeed, using degeneracy, we see that
\[D_i=\sum_{\substack{J\in\D_p(n):\\ i\in J}}\bigl(W_J^{(i)}-W_J\bigr)\]
is the Hoeffding decomposition of $D_i$, from which we read off that the Hoeffding component of $D_i$ belonging to a $p$-subset $J$ of $[n]$ is given by $-W_J$, whereas the Hoeffding component belonging to a $p$-subset $K$ of $[n+1]$ with $n+1\in K$ is given by $W^{(i)}_{(K\setminus\{n+1\})\cup\{i\}}$, if $i\notin K$ and equals $0$, if $i\in K$.

Hence, by considering an independent copy $(Y_1,\dotsc,Y_n,Z_{n+1})$ of $(X_1,\dotsc,X_n,X_{n+1})$  (for which only another copy $Z_{n+1}$ of $X_{n+1}=Y_i$ must be additionally chosen) and an independent, uniformly distributed index $\beta$ with values in $[n+1]$, in the same way as for $W$ itself, we construct an exchangeable pair $(D_i,D_i')$ which, by \eqref{linreg}, satisfies 
\[\E\bigl[D_i'-D_i\bigl|D_i\bigr]=-\frac{p}{n+1}D_i.\]
Thus, applying first the second bound in Lemma \ref{le3} to the pairs $(D_i,D_i')$, then the definition of $D_i$ and finally Lemma \ref{le2}, we obtain that 
\begin{align}\label{pr7}
&\sum_{\substack{1\leq i,j\leq n:\\i\not=j}}\E\Bigl[\bigl(D_i^{(j)}-D_i\bigr)^4\Bigr]
=\sum_{i=1}^n\Biggl(\sum_{\substack{1\leq j\leq n:\\ j\not=i}}\E\Bigl[\bigl(D_i^{(j)}-D_i\bigr)^4\Bigr]\Biggr)\notag\\
&\leq(n+1)\sum_{i=1}^n \E\Bigl[\bigl(D_i'-D_i\bigr)^4\Bigr]\leq 8p\sum_{i=1}^n \E\bigl[D_i^4\bigr]=8pn\E\Bigl[\bigl(W'-W\bigr)^4\Bigr]\notag\\
&\leq  64p^2\bigl(\E[W^4]-3\bigr)+96p^2\kappa_p \rho_n^2.
\end{align}
Similarly, using the first bound from Lemma \ref{le3} to the pairs $(D_i,D_i')$ this time instead, we obtain
\begin{align}\label{pr8}
&\sum_{\substack{1\leq i,j\leq n:\\i\not=j}}\E\Bigl[D_i^2\bigl(D_i^{(j)}-D_i\bigr)^2\Bigr]=\sum_{i=1}^n\Biggl(\sum_{\substack{1\leq j\leq n:\\ j\not=i}}\E\biggl[D_i^2\E\Bigl[\bigl(D_i^{(j)}-D_i\bigr)^2\,\bigl|D_i\Bigr]\biggr]\Biggr)\notag\\
&\leq (n+1)\sum_{i=1}^n\E\biggl[D_i^2\E\Bigl[\bigl(D_i'-D_i\bigr)^2\,\bigl|D_i\Bigr]\biggr]
=2p\sum_{i=1}^n\,\E\biggl[D_i^2 \frac{n+1}{2p}\E\Bigl[\bigl(D_i'-D_i\bigr)^2\,\bigl|D_i\Bigr]\biggr]\notag\\
&\leq2p \sum_{i=1}^n \E\bigl[D_i^4\bigr]=2pn\E\Bigl[\bigl(W'-W\bigr)^4\Bigr]
\leq  16p^2\bigl(\E[W^4]-3\bigr)+24p^2\kappa_p \rho_n^2.
\end{align}
Thus, \eqref{pr6}-\eqref{pr8} together imply that 
\begin{align}\label{pr9}
&\sum_{\substack{1\leq i,j\leq n:\\i\not=j}}\Cov\Bigl(\E\bigl[\theta(D_i)\,\bigl|\, X\bigr], \E\bigl[\theta(D_j)\,\bigl|\, X\bigr]\Bigr)
\leq 64p^2\bigl(\E[W^4]-3\bigr)+96p^2\kappa_p \rho_n^2
\end{align}
so that \eqref{pr3}, \eqref{pr4} and \eqref{pr9} yield that 
\begin{align}\label{pr10}
&\frac{n}{p}\E\Babs{\E\bigl[|W'-W|(W'-W)\,\bigl|\,W\bigr]}\notag\\
&\leq\frac{1}{p}\Bigl((8p+64p^2)\bigl(\E[W^4]-3\bigr)+(12p+96p^2)\kappa_p \rho_n^2\Bigr)^{1/2}\notag\\
&\leq \bigl(2\sqrt{2}p^{-1/2}+8\bigr)\sqrt{|\E[W^4]-3|}+\bigl(2\sqrt{3}p^{-1/2}+4\sqrt{6}\sqrt{\kappa_p}\bigr)\rho_n\notag\\
&\leq \bigl(2\sqrt{2}+8\bigr)\sqrt{|\E[W^4]-3|}+\bigl(2\sqrt{3}+4\sqrt{6}\sqrt{\kappa_p}\bigr)\rho_n .
\end{align}
Theorem \ref{maintheo} now follows from \eqref{pr1}, \eqref{pr2} and \eqref{pr10} by minor simplifications.
\end{proof}

\normalem
\bibliography{bedj}{}
\bibliographystyle{alpha}
\end{document}